\documentclass[10]{amsart} %The AMS standard layout

\usepackage{amsthm}
\usepackage{amsmath}
\usepackage{amsfonts}
\usepackage{latexsym}
\usepackage[T1]{fontenc}
\usepackage[latin1]{inputenc}

\newtheorem{theorem}{Theorem}[section]
\newtheorem{lemma}[theorem]{Lemma}

\newtheorem{example}{Example}[section]

\theoremstyle{definition}

\theoremstyle{remark}
\newtheorem{remark}{Remark}[section]

\begin{document}

%\begin{frontmatter} %Automatica style

%\title{\vspace*{-1.5cm}
%\begin{flushright}
%\begin{minipage}{4.7cm}
%\tiny {\sc Dept. of Math./CMA \hfill University of Oslo\\
%Pure Mathematics \hfill    no 3\\
%ISSN 0806--2439 \hfill    April 2012}
%\end{minipage}
%\end{flushright}
%\vskip1cm
%Maximum principles for jump diffusion processes with infinite horizon}

\title{Maximum principles for jump diffusion processes with infinite horizon}

%\title{Maximum principles for jump diffusion processes with infinite horizon \thanksref{footnoteinfo}} %Automatica style
%\title{Maximum principles for jump diffusion processes with infinite horizon} 

%\thanks[footnoteinfo]{The research leading to these results has received funding from the European Research %Council under the European Community's Seventh Framework %Programme (FP7/2007-2013) / ERC grant agreement no %[228087].\newline
%2010 Mathematics Subject Classification: \newline
%Primary 93EXX; 93E20; 60J75 \newline
%Secondary 60H10; 60H20; 49J55.} %Automatica style

\thanks{The research leading to these results has received funding from the European Research Council under the European Community's Seventh Framework Programme (FP7/2007-2013) / ERC grant agreement no [228087].} %AMS style

%\author{Sven Haadem}\ead{sven.haadem@cma.uio.no}, %Automatica style
%\author{Bernt Øksendal}\ead{oksendal@math.uio.no}, %Automatica style
%\author{Frank Proske} \ead{proske@math.uio.no} %Automatica style

\author{Sven Haadem} %AMS style
\author{Bernt Øksendal} %AMS style
\author{Frank Proske} %AMS style

%\address{Center of Mathematics for Applications (CMA), University of Oslo,Box 1053 Blindern, N-0316 Oslo, Norway} %Automatica style

\date{08 May 2012}

%\begin{keyword}
%Optimal control; L\'{e}vy processes; Maximum principle; Hamiltonian; Infinite horizon; Adjoint process; Partial information 
%\end{keyword}%Automatica style

\subjclass[2010]{Primary primary class 93EXX; 93E20; 60J75; Secondary secondary classes 60H10; 60H20; 49J55} %AMS style

\keywords{Optimal control; L\'{e}vy processes; Maximum principle; Hamiltonian; Infinite Horizon; Adjoint Process; Partial Information} %AMS style

%\author{Sven Haadem\footnote{Center of Mathematics for Applications (CMA), University of Oslo,
%Box 1053 Blindern, N-0316 Oslo, Norway. Email: sven.haadem@cma.uio.no},
%Frank Proske\footnote{Center of Mathematics for Applications (CMA), University of Oslo,
%Box 1053 Blindern, N-0316 Oslo, Norway. Email: proske@math.uio.no} and
%Bernt Øksendal\footnote{Center of Mathematics for Applications (CMA), University of Oslo,
%Box 1053 Blindern, N-0316 Oslo, Norway.
%The research leading to these results has received funding from the European Research Council under the European Community's Seventh %Framework Programme (FP7/2007-2013) / ERC grant agreement no [228087].
%Email: oksendal@math.uio.no}.
%} %Pure Latex style

\date{8 May 2012}

%\textsf{Keywords: Optimal control; L\'{e}vy processes; Maximum principle; Hamiltonian; Infinite horizon; Adjoint process; Partial %information\newline\newline
%2010 Mathematics Subject Classification: \newline
%Primary 93EXX; 93E20; 60J75 \newline
%Secondary 60H10; 60H20; 49J55} %Pure Latex style

\begin{abstract}
We prove maximum principles for the problem of optimal control for a jump diffusion with infinite horizon and partial information. The results are applied to partial information optimal consumption and portfolio problems in infinite horizon.
\end{abstract}

\maketitle

%\end{frontmatter} %Automatica style

\section{Introduction}

In this paper  we consider a control problem for a performance functional
\begin{align*}
J(u) = E\left[ \int_0^{\infty} f(t, X(t),u(t), \omega)dt\right],
\end{align*}
where $X(t)$ is a controlled jump diffusion and $u(t)$ is the control process. We allow for the case where the controller only has access to partial-information. Thus, we have a infinite horizon problem with partial information. Infinite-horizon optimal control problems arise in many fields of economics, in particular in models of economic growth. Note that because of the general nature of the partial information filtration $\mathcal{E}_t$, we cannot use dynamic programming and Hamilton-Jacobi-Bellman (HJB) equations to solve the optimization problem. Thus our problem is different from partial observation control problems. 
\newline\newline
In the deterministic case the maximum principle by Pontryagin (1962) has been extended to infinite-horizon problems, but transversality conditions have not been given in gerneral. 
The 'natural' transversality condition in the infinite case would be a zero limit condition, meaning in the economic sense that one more unit of good at the limit gives no additional value. But this property is not necessarily verified. In fact \cite{halkin} provides a counterexample for a 'natural' extension of the finite-horizon transversality conditions. Thus some care is needed in the infinite horizon case.
\newline\newline
There have been a variety of articles on infinite-horizon problems. E.g. in \cite{Ververka} it is stated a 'natural' extension to infinite horizon discounted control problems.
\newline\newline
We refer to \cite{Bernt2} for more information about stochastic control in jump diffusion
markets, to \cite{Peng} for a background on infinite-horizon backward stochastic differential equations and \cite{Sydseter} for a general introduction to infinite-horizon control problems in a deterministic environment. \newline\newline
In this paper we prove several maximum principles for an infinite horizon optimal control problem with partial information. The paper is structured as follows:
In Section 4 we prove a maximum principle version of sufficient type (a verification theorem). In section 5 we give some examples, before we prove a (weak) version of a necessary type of the maximum principle in section 6. \newline\newline
In a forthcomming paper \cite{AHOP}, the case of infinite horizon for delay equations is treated.

\section{Preliminaries}
Let $B(t) = B(t,\omega) = (B_1(t,\omega),\ldots,B_n(t,\omega))$, $t \geq0$, $\omega \in \Omega$ and $\tilde{N}(dz,dt) = N(dz,dt) - \nu(dz)dt = (\tilde{N}_1(dz,dt),\ldots,\tilde{N}_n(dz,dt))$ be a n-dimensional  Brownian motion and n independent compensated Poisson random measures, respectively, on a filtered probability space $(\Omega, \mathcal{F}, \{\mathcal{F}_t\}_{t\geq 0},P)$.
Let $X(t) = X^u(t)$ be a controlled jump diffusion, described by the stochastic differential equation
\begin{align}\label{eq.X}
dX(t) &= b(t,X(t),u(t), \omega)dt + \sigma(t,X(t),u(t), \omega)dB(t) \notag\\
&+ \int_{\mathbb{R}_0^n} \theta(t,X(t),u(t),z, \omega)\tilde{N}(dz,dt); 0 \leq t < \infty \\
X(0) &= x \in \mathbb{R}^n,\notag
\end{align}
where $b:[0,\infty]\times \mathbb{R}^n \times U \times \Omega \rightarrow \mathbb{R}^n$ is adapted, $\sigma:[0,\infty]\times \mathbb{R}^n \times U \times \Omega \rightarrow  \mathbb{R}^{n\times n}$ is adapted and  $\theta:[0,\infty]\times \mathbb{R}^n \times U \times \Omega \rightarrow \mathbb{R}^{n\times n}$ is predictable (see \cite{Rogers}). See e.g. \cite{Bernt}, \cite{Bernt2} for notation and more information.
Let
\begin{align*}
 \mathcal{E}_t \subset \mathcal{F}_t,
\end{align*}
be a given subfiltration, representing the information available to the controller at time $t; t \geq0$.
The process $u(t)$ is our control, assumed to be $\{\mathcal{E}_t\}_{t\geq0}$ predicatble and with values in a set $U \subset \mathbb{R}^n$.
Let $\mathcal{A}_{\mathcal{E}}$ be our family of $\mathcal{E}_t$-predicatble controls.
Let $\mathcal{R}$ denote the set of functions $r:[0,\infty] \times \mathbb{R}_0^n \rightarrow \mathbb{R}^{n \times n} $ such that
\begin{align*}
\int_{\mathbb{R}_0^n} |\theta_{i,j}(t,x,u,z)r_{i,j}(t,z)|\nu_j(dz) < \infty \text{ for all } i,j,t,x.  
\end{align*}
Let $f:[0,\infty]\times \mathbb{R}^n \times U \times \Omega \rightarrow \mathbb{R}^n$ be adapted and assume that
\begin{align*}
E\left[ \int_0^{\infty} |f(t, X(t),u(t), \omega)|dt\right] < \infty \text{ for all } u\in \mathcal{A}_{\mathcal{E}}.\\
\end{align*}
Then we define
\begin{align*}
J(u) = E\left[ \int_0^{\infty} f(t, X(t),u(t), \omega)dt\right]
\end{align*}
to be our performance functional.
We study the problem to find $\hat{u}\in \mathcal{A}_{\mathcal{E}}$ such that
\begin{align}\label{eq.opt}
J(\hat{u}) = \sup_{u\in \mathcal{A}_{\mathcal{E}}}J(u).
\end{align}
Let us define the Hamiltonian $H:[0,T] \times \mathbb{R}^n \times U \times \mathbb{R}^n \times \mathbb{R}^{n \times n} \times \mathcal{R} \rightarrow \mathbb{R}$, by
\begin{align}
H(t,x,u,p,q,r)  &= f(t,x,u, \omega) + b^{T}(t,x,u, \omega)p + tr(\sigma^{T}(t,x,u, \omega)q) \notag \\
&+ \sum_{i,j=1}^{n} \int_{\mathbb{R}_0^n} \theta_{i,j}(t,x,u,z, \omega)r_{i,j}(t,z)\nu_{j}(dz)\label{eq:hamiltonian}. 
\end{align}
For notational convenience we will in the rest of the paper suppress any $\omega$ from the notation.
The adjoint equation in the unknown $\mathcal{F}_t$-predictable processes $(p(t),q(t),r(t,z))$ is the following
\begin{align}\label{eq:adjoint} 
dp(t) &= - \nabla_x H(t,X(t),\hat{u}(t),p(t),q(t),r(t,\cdot))dt + q(t)dB(t)\notag\\
&+ \int_{\mathbb{R}_0^n}r(t,z)\tilde{N}(dz,dt).
\end{align}

\section{Existence and Uniqueness}
In this section we prove a result about existence and uniqueness of the solution $(Y(t),Z(t),K(t,\zeta))$ of  infinite horizon BSDEs of the form;
\begin{align}
dY(t) &= -g(t,Y(t),Z(t),K(t,\cdot))dt + Z(t) dB(t) \notag \\
& + \int_{\mathbb{R}_0^n}K(t,\zeta)\tilde{N}(d\zeta,dt); 0\leq t\leq \tau, \label{eq.existence1}\\
\underset{t \to \tau}{\lim} Y(t) &= \xi(\tau)\mathbf{1}_{[0,\infty)}(\tau), \label{eq.existence2}
\end{align}
where $\tau \leq \infty$ is a given $\mathcal{F}_t$-stopping time, possibly infinite.
Our result is an extension to jumps of Theorem 4.1 in \cite{Pardoux}, Theorem 4 in \cite{Peng} and Theorem 3.1 in \cite{Yin}. It is also an extension to infinite horizon of Theorem Lemma 2.1 in \cite{Li}. See also \cite{Tang}, \cite{BSDE}, \cite{BBP} and \cite{Rong}.
We assume the following:
\begin{enumerate}
  \item The function $g: \Omega  \times \mathbb{R}_{+} \times \mathbb{R}^{k} \times \mathbb{R}^{k\times d} \times \mathcal{R} \to \mathbb{R}^{k} $ is such that there exist real numbers $\mu,\lambda, K_1$ and $K_2$, such that $K_1,K_2 >0$ and
\begin{align}\label{eq:lambda_req}
\lambda > 2\mu +K_1^2 + K_2^2.
\end{align}
We assume that the function $g$ satisfies the following requirements:
\begin{enumerate}
	\item $g(\cdot,y,z,k)$ is progessively measurable for all $y,z,k$, and
\begin{align}
 &|g(t,y,z,k(\cdot)) - g(t,y,z',k'(\cdot))| \leq K_1 \lVert z - z'\rVert \notag\\
&+ K_2 \lVert k(\cdot) - k'(\cdot)\rVert_{R},
\end{align}
where
\[
\lVert k(\cdot)\rVert_{R}^2 =  \int_{\mathbb{R}_0^n}  k^2(\zeta)\nu(d\zeta),
\] 
and $\lVert z\rVert = [Tr(zz^*)]^{\frac{1}{2}}$.
	\item 
\begin{align}
\langle y-y',g(t,y,z,k) - g(t,y',z,k)\rangle \leq \mu |y-y'|^2
\end{align}
for all $y,y',z,k$ a.s.
	\item 
\begin{align}
E \int_0^{\tau} e^{\lambda t} |g(t,0,0,0)|^2dt < \infty.
\end{align}
	\item  Finaly we require that
\begin{align}
y \mapsto g(t,y,z,k),
\end{align}
is continuous for all $t,z,k$ a.s.
	\end{enumerate}
  \item We have a final condition $\xi$, which is $\mathcal{F}_{\tau}$-measurable such that\\
$E(e^{\lambda \tau}|\xi|^2) < \infty $ and 
\begin{align}
 E \int_0^{\tau} e^{\lambda t} |g(t,\xi_t,\eta_t,\psi_t)|^2dt < \infty,
\end{align}
where $\xi_t = E(\xi|\mathcal{F}_t)$ and $\eta$,$\psi$ are s.t.
\begin{align}
 \xi = E\xi + \int_0^t \eta(s) dB_s + \int_0^t \int_{\mathbb{R}_0^n} \psi(s,\zeta) \tilde{N}(d\zeta,ds).
\end{align}
\end{enumerate}
A solution of the BSDE \eqref{eq.existence1}-\eqref{eq.existence2}, is a trippel $(Y_t,Z_t,K_t)$ of progressively measurable processes with values in $\mathbb{R} \times \mathbb{R} \times \mathbb{R}$ s.t. $Z_t$, $K_t = 0$ when $t > \tau$, 
\begin{enumerate}
	\item $E[\underset{t \geq 0}{\sup } \text{ } e^{\lambda t} |Y_t|^2 + \int_0^{\tau}e^{\lambda s} \lvert Z_s \rvert^2 ds + \int_{0}^{\tau} \int_{\mathbb{R}_0^n} e^{\lambda s} K^2(s,\zeta)\nu(d\zeta)ds] < \infty$,
	\item  $Y_t = Y_{T \wedge \tau} + \int_{t \wedge \tau}^{T \wedge \tau} g_s ds - \int_{t \wedge \tau}^{T \wedge \tau} Z_s dB_s - \int_{t \wedge \tau}^{T \wedge \tau}\int_{\mathbb{R}_0^n}K(s,\zeta)\tilde{N}(d\zeta,ds) $ for all deterministic $T < \infty$ and
	\item $Y_t = \xi$ on the set $\{t \geq \tau\}$.
\end{enumerate}
\begin{remark}[Infinite Horizon]
This incorperates the case where $\tau(\omega) = \infty$ on some set $A$ with $P(A) > 0$, possibly $P(A)=1$.
\end{remark}

\begin{theorem}[Existence and uniqueness]
Under the above conditions there exists a unique solution $(Y_t,Z_t,K_t)$ of the BSDE \eqref{eq.existence1}-\eqref{eq.existence2}, which satisfies the condition;
\begin{align}
&E [  \underset{0 \leq t \leq \tau}{\sup} e^{\lambda t} | Y_{t } |^2 + \int_{0}^{\tau} e^{\lambda s} (|Y_s|^2 + \parallel Z_s \parallel^2)ds + \int_{0}^{\tau} e^{\lambda s} \int_{\mathbb{R}_0^n} K^2(s,\zeta)\nu(d\zeta)ds]\notag \\ 
&\leq c E[e^{\lambda \tau}|\xi|^2 + \int_0^{\tau}e^{\lambda s}|g(s,0,0,0)|^2ds]\label{eq.Req},
\end{align}
for some positive number $c$.
\end{theorem}

\begin{proof}
First, let us show uniqueness. Let $(Y,Z,K)$ and $(Y',Z',K')$ be two solutions satisfying \eqref{eq.Req} and let $(\bar{Y},\bar{Z},\bar{K}) = (Y-Y',Z-Z',K-K')$. From It$\bar{o}$'s Lemma we have that
\begin{align*}
e^{\lambda t \wedge \tau} |\bar{Y}_{t \wedge \tau} |^2 &+ \int_{t \wedge \tau}^{T \wedge \tau} \Bigg[e^{\lambda s} ( \lambda |\bar{Y}_s|^2 + \lVert \bar{Z}_s \rVert^2)
+  e^{\lambda s}\int_{\mathbb{R}_0^n}  \bar{K}^2(s,\zeta)\nu(d\zeta)\Bigg]ds\\ 
&\leq e^{\lambda s}|\bar{Y}_{T}|^2 + 2\int_{t \wedge \tau}^{T \wedge \tau} \Bigg[e^{\lambda s} (\mu |\bar{Y}_s|^2 + K_1  |\bar{Y}_s|\times\lVert \bar{Z}_s \rVert) \\
&+ K_2|\bar{Y}_s|e^{\lambda s}(\int_{\mathbb{R}_0^n}  \bar{K}^2(s,\zeta)\nu(d\zeta))^{\frac{1}{2}}\Bigg]ds\\ 
&- 2\int_{t \wedge \tau}^{T \wedge \tau} e^{\lambda s} \langle \bar{Y}_s,\bar{Z}_s dB_s \rangle\\ 
&- \int_{t \wedge \tau}^{T \wedge \tau} e^{\lambda s}\int_{\mathbb{R}_0^n} \left[\bar{K}^2(s,\zeta) + 2\bar{K}(s,\zeta)  \bar{Y}(s) \right] \tilde{N}(d\zeta,ds).
\end{align*}
Combining the above with the fact that $2ab \leq a^2 + b^2$ we deduce since $\lambda > 2\mu +K_1^2 + K_2^2$, that for $t < T$
\[
E [ e^{\lambda t \wedge \tau} |\bar{Y}_{t \wedge \tau} |^2] \leq E[e^{\lambda T \wedge \tau} |\bar{Y}_{T} |^2]
\]
the same holds with $\lambda$ replaced by $\lambda^{'}$, with $\lambda > \lambda' > 2\mu +K_1^2 + K_2^2$ 
\[
E \Big[ e^{\lambda t \wedge \tau} |\bar{Y}_{t \wedge \tau} |^2\Big] \leq e^{(\lambda-\lambda')T}E \Big[e^{\lambda T \wedge \tau} |\bar{Y}_{T} |^2 \mathbf{1}_{\{T < \tau\}}\Big]
\]
Condition \eqref{eq.Req} implies that the second factor on the right hand side remains bounded as $T \to \infty$, while the first factor tends to $0$. This proves uniqueness.\newline\newline
\textit{Proof of existence}. For each $n \in \mathrm{N}$ we construct a solution $(Y_t^n,Z^n_t,K^n_t) $ of the BSDE 
\[
 Y_t^n = \xi + \int_{t \wedge \tau}^{n \wedge \tau} g(s,Y_s^n,Z_s^n,K^n_s)ds 
-\int_{t \wedge \tau}^{\tau}  Z^n_s dB_s 
- \int_{t \wedge \tau}^{\tau}  \int_{\mathbb{R}_0^n}K^n(s,\zeta)\tilde{N}(d\zeta,ds) 
\]
by letting $\{(Y_t^n,Z^n_t,K^n_t); 0 \leq t\leq n \}$  be defined as a solution of the following BSDE:
\[
 Y_t^n = E[\xi|\mathcal{F}_n] + \int_{t }^{n } \mathbf{1}_{[0,\tau]}(s)g(s,Y_s^n,Z_s^n, K^n_s)ds
-\int_{t }^{n }  Z^n_s dB_s 
- \int_{t }^{n }  \int_{\mathbb{R}_0^n}K^n(s,\zeta)\tilde{N}(d\zeta,ds) 
\]
for $0\leq t\leq n$ and $\{(Y_t^n,Z^n_t,K^n_t); t\geq n \}$ defined by
\[
 Y_t^n = \xi_t,
\]
\[
 Z^n_t = \eta_t, 
\]
and
\[
 K^n_t = \psi_t, 
\]
for $t> n$. Next, we find some a priori estimates for the sequence $(Y^n,Z^n,K^n)$. For any $\epsilon > 0$, $\rho < 1$ and $\alpha$ we have for all $t \geq 0, y \in \mathbb{R}^{k}$, $z\in \mathbb{R}^{k\times d}$, $k \in \mathcal{R}$  with $c= \frac{1}{\epsilon}$,
\begin{align*}
&2 \langle y,g(t,y,z,k)\rangle = 2 \langle y,g(t,y,z,k) - g(t,0,z,k) \rangle \\
&+ 2 \langle y,g(t,0,z,k) - g(t,0,0,0) \rangle + 2 \langle y,g(t,0,0,0) \rangle  \\
&\leq (2\mu + \frac{1}{\rho}K_1^2 + \frac{1}{\alpha}K_2^2 + \epsilon)|y|^2
+ \rho \parallel z \parallel^2 + \alpha\int_{\mathbb{R}_0^n}  k^2(\zeta)\nu(d\zeta)\\
&+ c|g(t,0,0,0)|^2.
\end{align*}
From It$\bar{o}$'s Lemma we have
\begin{align*}
e^{\lambda t \wedge \tau} |Y_{t \wedge \tau}^n |^2 &+ \int_{t \wedge \tau}^{\tau} \left[e^{\lambda s} ( \bar{\lambda} |Y_s^n|^2 + \bar{\rho} \parallel Z_s^n \parallel^2 )
+ \bar{\alpha}\int_{t \wedge \tau}^{\tau} e^{\lambda s}\int_{\mathbb{R}_0^n}  (K^n)^2(s,\zeta)\nu(d\zeta) \right]ds\\ 
&\leq e^{\lambda s}|\eta|^2 + c\int_{t \wedge \tau}^{\tau} e^{\lambda s} |g(s,0,0,0)|^2ds\\ 
&- 2\int_{t \wedge \tau}^{\tau} e^{\lambda s} <Y_s^n,Z_s^ndB_s>\\
&- \int_{t \wedge \tau}^{\tau} e^{\lambda s}\int_{\mathbb{R}_0^n} \left[(K^n)^2(s,\zeta) + 2K^n(s,\zeta)  Y^n(s) \right] \tilde{N}(d\zeta,ds),
\end{align*}
with $\bar{\lambda} = \lambda - 2\mu - \frac{1}{\rho}K_1^2 - \frac{1}{\alpha}K_2^2 - \epsilon > 0$, $\bar{\rho} = 1-\rho >0$ and $\bar{\alpha} = 1 - \alpha$. 
From this and the matingale inequality it follows that
\begin{align*}
&E \left[ \underset{t\geq s}{\sup}\text{ } e^{\lambda t \wedge \tau} |Y_{t \wedge \tau}^n |^2 + \int_{s \wedge \tau}^{\tau} \Big[e^{\lambda r} (|Y_r^n|^2 + \parallel Z_r^n \parallel^2) 
+  e^{\lambda r}\int_{\mathbb{R}_0^n}  (K^n)^2(r,\zeta)\nu(d\zeta)\Big]dr\right]\\
&\leq D E\left[ e^{\lambda \tau}|\xi|^2 + \int_{s \wedge \tau}^{\tau} e^{\lambda r}|g(r,0,0,0)|^2dr \right].
\end{align*}
Let $m > n$ and define $\Delta Y_t := Y_t^m - Y_t^n$, $\Delta Z_t := Z_t^m - Z_t^n$ and $\Delta K_t := K_t^m - K_t^n$, so that for $n \leq t \leq m$,
\[
 \Delta Y_t = \int_{t \wedge \tau}^{m \wedge \tau} g(s,Y_s^m,Z_s^m,K_s^m) ds
-\int_{t \wedge \tau}^{m \wedge \tau} \Delta Z_s dB_s 
- \int_{t \wedge \tau}^{m \wedge \tau}  \int_{\mathbb{R}_0^n} \Delta K(s,\zeta)\tilde{N}(d\zeta,ds).
\]
It then follows that
\begin{align*}
e^{\lambda t \wedge \tau} |\Delta Y_{t \wedge \tau} |^2 &+ \int_{t \wedge \tau}^{m \wedge \tau} \Big\{e^{\lambda s} ( \lambda |\Delta Y_s|^2 +  \parallel \Delta Z_s \parallel^2) 
+  e^{\lambda s}\int_{\mathbb{R}_0^n}  (\Delta K)^2(s,\zeta)\nu(d\zeta)\Big\}ds\\ 
& = \int_{t \wedge \tau}^{m \wedge \tau}  e^{\lambda s} \langle \Delta Y_s,g(s,Y_s^m,Z_s^m,K_s^m) \rangle ds\\
&- 2 \int_{t \wedge \tau}^{m \wedge \tau}  e^{\lambda s} \langle \Delta Y_s,\Delta Z_s dB_s \rangle \\
&- \int_{t \wedge \tau}^{m \wedge \tau} e^{\lambda s}\int_{\mathbb{R}_0^n}\left[(\Delta K)^2(s,\zeta) + 2\Delta K(s,\zeta)\Delta Y(s) \right]\tilde{N}(d\zeta,ds)\\
&2\leq e^{\lambda s}|\eta|^2 c\int_{t \wedge \tau}^{m \wedge \tau} e^{\lambda s} |g(s,0,0,0)|^2ds 
- 2\int_{t \wedge \tau}^{m \wedge \tau} e^{\lambda s} \langle \Delta Y_s,\Delta Z_s dB_s \rangle \\ 
&- \int_{t \wedge \tau}^{m \wedge \tau} e^{\lambda s}\int_{\mathbb{R}_0^n} \left[(\Delta K)^2(s,\zeta) + 2\Delta K(s,\zeta)\Delta Y(s) \right]\bar{N}(d\zeta,ds).
\end{align*}
From the same arguments as above
\begin{align*}
&E \Bigg[  \underset{n \leq t \leq m}{\sup} e^{\lambda t \wedge \tau} |\Delta Y_{t \wedge \tau} |^2 \\
&+ \int_{n \wedge \tau}^{m \wedge \tau} \Big\{e^{\lambda s} (|\Delta Y_s|^2 + \parallel \Delta Z_s \parallel^2) +  e^{\lambda s}\int_{\mathbb{R}_0^n}  (\Delta K)^2(s,\zeta)\nu(d\zeta)\Big\}ds\Bigg]\\
&\leq 4E\left[\int_{n \wedge \tau}^{\tau} e^{\lambda s}|g(s,\xi,\eta,\psi)|^2ds\right]. 
\end{align*}
The last term in the above equation goes to zero as $n \to \infty$. Now, for $t\leq n$
\begin{align*}
\Delta Y_t &= \Delta Y_n + \int_{t \wedge \tau}^{n \wedge \tau} \Big\{g(s,Y_s^m,Z_s^m,K_s^m) - g(s,Y_s^n,Z_s^n,K_s^n)\Big\} ds -\int_{t \wedge \tau}^{n \wedge\tau} \Delta Z_s dB_s\\
&- \int_{t \wedge \tau}^{n \wedge \tau}  \int_{\mathbb{R}_0^n}\Delta K(s,\zeta)\tilde{N}(d\zeta,ds). 
\end{align*}
Using the same argument as in the case of uniqueness, we have that
\begin{align*}
 E[e^{\lambda t \wedge \tau} |\Delta Y_{t\wedge \tau}|^2] 
\leq E[e^{\lambda t \wedge \tau} |\Delta Y_n|^2] 
\leq c E\left[ \int_{n \wedge \tau}^{\tau} e^{\lambda s}|g(s,\xi_s,\eta_s,\psi_s)|^2ds\right].
\end{align*}
It now follows that the sequence $(Y^n,Z^n, K^n)$ is Cauchy in the norm
\begin{align*}
\lVert (Y,Z,K) \rVert &:= E [  \underset{0 \leq t \leq \tau}{\sup} e^{\lambda t} | Y_{t } |^2 + \int_{0}^{\tau} e^{\lambda s} (|Y_s|^2 + \parallel Z_s \parallel^2)ds\\
&+ \int_{0}^{\tau} e^{\lambda s}\int_{\mathbb{R}_0^n}  K^2(s,\zeta)\nu(d\zeta)ds].
\end{align*}
So, we have that there is an unique solution to the BSDE \eqref{eq.existence1}-\eqref{eq.existence2}, which satisfies for all  $\lambda > 2\mu +K_1^2 + K_2^2$, the condition
\begin{align*}
&E \left[  \underset{0 \leq t \leq \tau}{\sup} e^{\lambda t} | Y_{t } |^2 + \int_{0}^{\tau} e^{\lambda s} (|Y_s|^2 + \parallel Z_s \parallel^2)ds + \int_{0}^{\tau} e^{\lambda s}\int_{\mathbb{R}_0^n}  K^2(s,\zeta)\nu(d\zeta)ds\right]\\ 
&\leq c E\left[e^{\lambda \tau}|\xi|^2 + \int_0^{\tau}e^{\lambda s}|g(s,0,0,0)|^2ds\right].
\end{align*}

\end{proof}

\section{Optimal control with partial information and infinite horizon}
Now, let us get back to the problem of maximizing the performance functional
\begin{align*}
J(u) = E\left[ \int_0^{\infty} f(t, X(t),u(t))dt\right],
\end{align*}
where $X(t)$ is of the form \eqref{eq.X}. Our aim is to find a $\hat{u}\in \mathcal{A}_{\mathcal{E}}$ such that
\begin{align*}
J(\hat{u}) = \sup_{u\in \mathcal{A}_{\mathcal{E}}}J(u),
\end{align*}
where $u(t)$ is our previsible control adapted to a subfiltration
\begin{align*}
 \mathcal{E}_t \subset \mathcal{F}_t,
\end{align*}
with values in a set $U \subset \mathbb{R}^n$.
Let $H$ be the Hamiltonian defined by \eqref{eq:hamiltonian} and $p$ the solution to the adjoint equation  \eqref{eq:adjoint}. Then we have the following maximum principle;
\begin{theorem}[Sufficient Infinite Horizon Maximum Principle]
Let $\hat{u} \in \mathcal{A}_{\mathcal{E}}$ and let  $(\hat{p}(t),\hat{q}(t),\hat{r}(t,z))$ be an associated solution to the  equation \eqref{eq:adjoint}. Assume that for all $u\in \mathcal{A}_{\mathcal{E}}$ the following terminal condition holds:
\begin{align}\label{eq:lim}
0 \leq E\left[ \overline{\lim_{t \to \infty}} [ \hat{p}(t)^{T}(X(t) - \hat{X}(t)) ]\right] < \infty. 
\end{align}
Moreover, assume that $H(t,x,u,\hat{p}(t),\hat{q}(t),\hat{r}(t,\cdot))$ is concave in $x$ and $u$ and
\begin{align}
E\left[ H(t,\hat{X}(t),\hat{u}(t),\hat{p}(t),\hat{q}(t),\hat{r}(t,\cdot)) | \mathcal{E}_t\right] \notag \\
= \max_{u \in U} E\left[ H(t,\hat{X}(t),u,\hat{p}(t),\hat{q}(t),\hat{r}(t,\cdot)) | \mathcal{E}_t\right]. \label{eq:max}
\end{align}
In addition we assume that
\begin{align}
&E\left[\int_0^{\infty} (\hat{X}(t) - X^{u}(t))^T [\hat{q}\hat{q}^T + \int_{\mathbb{R}_0^n}\hat{r}\hat{r}^T(t,z) \nu (dz)] (\hat{X}(t) - X^{u}(t))dt\right] < \infty, \label{eq:mart1} \\
&E\left[\int_0^{\infty} \hat{p}(t)^T [\sigma\sigma^T(t,X(t),u(t)) + \int_{\mathbb{R}_0^n}\theta \theta^T(t,X(t),u(t)) \nu (dz)] p(t)dt\right] < \infty,   \label{eq:mart2} \\
&E\left[|\nabla_u H(t,\hat{X}(t),\hat{u}(t),\hat{p}(t),\hat{q}(t),\hat{r}(t,\cdot))|^2\right] < \infty, \label{eq:derivL2}
\end{align}
and that
\begin{align}
&E\left[ \int_0^{\infty} |H(s,X(s),u(s),\hat{p}(s),\hat{q}(s),\hat{r}(s,\cdot))|\right] < \infty
\end{align}
for all $u$.\newline
Then we have that $\hat{u}(t)$ is optimal.
\end{theorem}

\begin{remark}
Note that, since $p(t)$ has the economic interpretation as the marginal value of the resource (alternativly the shadow price if representing an outside resource), the requirement
\begin{align*}
0 \leq E\left[ \overline{\lim_{t \to \infty}} [ \hat{p}(t)^{T}(X(t) - \hat{X}(t)) ]\right] < \infty,
\end{align*}
has the economic interpretation that if the marginal value is positive at infinity we want to have as little resources left as possible.
\end{remark}

\begin{remark}
 The requirement in the finite horizon case that $p(T) = 0$ does not translate into $\underset{T \to \infty}{\lim p(T)} = 0$ as was shown in the deterministic case in \cite{halkin}.
\end{remark}

\begin{proof}
Let $I^{\infty} := E[ \int_0^{\infty} (f(t, X(t),u(t)) - f(t, \hat{X}(t),\hat{u}(t)))dt] = J(u) -J(\hat{u})$.
Then $I^{\infty} = I_{1}^{\infty} - I_{2}^{\infty} - I_{3}^{\infty} - I_{4}^{\infty}$, where
\begin{align}
I_{1}^{\infty} &:= E\Biggl[ \int_0^{\infty} (H(s,X(s),u(s),\hat{p}(s),\hat{q}(s),\hat{r}(s,\cdot)) \notag\\ &-H(t,\hat{X}(s),\hat{u}(t),\hat{p}(s),\hat{q}(s),\hat{r}(s,\cdot)))ds \Biggr] \label{eq:I11}, \notag\\
I_{2}^{\infty} &:= E\left[ \int_0^{\infty} \hat{p}(s)^T(b(s,X(s),u(s))-\hat{b}(s,\hat{X}(s),\hat{u}(s)))ds\right] \notag, \\
I_{3}^{\infty} &:= E\left[\int_0^{\infty} \text{tr}[q(s)^T (\sigma(s,X(s),u(s)) - \hat{\sigma}(s,\hat{X}(s),\hat{u}(s)))]ds\right] \notag, 
\end{align}
and
\begin{align}
I_{4}^{\infty} &:= E\Bigg[\int_0^{\infty} \sum_{i,j} \int_{\mathbb{R}_0^n} (\theta(s,X(s),u(s),z)\notag\\
&- \hat{\theta}(s,\hat{X}(s),\hat{u}(s),z))^T \hat{r}_{i,j}(s,z)\nu_{j}(dz)ds \Bigg]. \notag
\end{align}
We have from concavity that
\begin{align}
&H(t,X(t),u(t),\hat{p}(t),\hat{q}(t),\hat{r}(t,\cdot)) - H(t,\hat{X}(t),\hat{u}(t),\hat{p}(t),\hat{q}(t),\hat{r}(t,\cdot))\\
&\leq \nabla_x H(t,\hat{X}(t),\hat{u}(t),\hat{p}(t),\hat{q}(t),\hat{r}(t,\cdot))^T(X(t) - \hat{X}(t)) \notag \\
&+ \nabla_u H(t,\hat{X}(t),\hat{u}(t),\hat{p}(t),\hat{q}(t),\hat{r}(t,\cdot))^T(u(t) - \hat{u}(t)). \label{eq:concave}
\end{align}
Then we have from \eqref{eq:max},\eqref{eq:derivL2} and that $u(t)$ is adapted to $\mathcal{E}_t$,
\begin{align}
0 &\geq \nabla_u E\left[ H(t,\hat{X}(t),u,\hat{p}(t),\hat{q}(t),\hat{r}(t,\cdot)) | \mathcal{E}_t\right]^T_{u=\hat{u}(t)}(u(t) - \hat{u}(t)) \notag\\
&= E\left[ \nabla_u H(t,\hat{X}(t),\hat{u}(t),\hat{p}(t),\hat{q}(t),\hat{r}(t,\cdot))^T(u(t) - \hat{u}(t)) | \mathcal{E}_t\right]. \label{eq:equality}
\end{align}
Combining \eqref{eq:adjoint}, \eqref{eq:mart1}, \eqref{eq:I11}, \eqref{eq:concave} and \eqref{eq:equality}
\begin{align*}
I_{1}^{\infty} &\leq E\left[ \int_0^{\infty} \nabla_x H(t,\hat{X}(s),\hat{u}(s),\hat{p}(s),\hat{q}(s),\hat{r}(s,\cdot))^T(X(s) - \hat{X}(s)) ds\right]\\
&= E\left[ \int_0^{\infty} (X(s) - \hat{X}(s))^T d\hat{p}(s) \right] =: -J_1.
\end{align*}
Now, using \eqref{eq:lim} and It$\bar{o}$'s formula
\begin{align*}
0 &\leq E\left[ \overline{\lim_{t \to \infty}} [ \hat{p}(t)^{T}(X(t) - \hat{X}(t)) ]\right] \\
&= E\Biggl[\overline{\lim_{t \to \infty}} \Big[ \int_0^t \hat{p}(s)^T(b(s,X(s),u(s))-\hat{b}(s,\hat{X}(s),\hat{u}(s)))ds\\ 
&+ \int_0^t \hat{p}(s)^T(\sigma(s,X(s),u(s)) - \hat{\sigma}(s,\hat{X}(s),\hat{u}(s))) dB(s)\\
&+ \int_0^t \int_{\mathbb{R}_0^n} \hat{p}(s)^T (\theta(s,X(s),u(s),z) - \hat{\theta}(s,\hat{X}(s),\hat{u}(s),z))\tilde{N}(dz,ds)\\
&+  \int_0^t (X(s) - \hat{X}(s))^T (-\nabla_x \hat{H}(s,\hat{X}(s),\hat{u}(s),\hat{p}(s),\hat{q}(s),\hat{r}(s,\cdot)))ds\\ 
&+ \int_0^t \hat{q}(s)^T(X(s) - \hat{X}(s)) dB(s)\\
&+ \int_0^t \int_{\mathbb{R}_0^n}\hat{r}(s,z)(X(s) - \hat{X}(s))\tilde{N}(dz,ds)\\ 
&+ \int_0^t \text{tr}\left[\hat{q}(s)^T (\sigma(s,X(s),u(s)) - \hat{\sigma}(s,\hat{X}(s),\hat{u}(s)))\right]ds\\
&+ \int_0^t \sum_{i,j} \int_{\mathbb{R}_0^n} (\theta(s,X(s),u(s),z) - \hat{\theta}(s,\hat{X}(s),\hat{u}(s),z))^T \hat{r}_{i,j}(s,z)\nu_{j}(dz)ds\\
&+ \int_0^t  \int_{\mathbb{R}_0^n} (\theta(s,X(s),u(s),z) - \hat{\theta}(s,\hat{X}(s),\hat{u}(s),z))^T \hat{r}(s,z)\tilde{N}(dz,ds) \Big]\Biggr]\\
\end{align*}
From \eqref{eq:mart1}, \eqref{eq:mart2}, we have that
\begin{align*}
0 &\leq E\Biggl[\overline{\lim_{t \to \infty}} \Big[ \int_0^t \hat{p}(s)^T(b(s,X(s),u(s))-\hat{b}(s,\hat{X}(s),\hat{u}(s)))ds\\ 
&+  \int_0^t (X(s) - \hat{X}(s))^T (-\nabla_x \hat{H}(s,\hat{X}(s),\hat{u}(s),\hat{p}(s),\hat{q}(s),\hat{r}(s,\cdot)))ds \\
&+ \int_0^t \text{tr}\left[\hat{q}(s)^T (\sigma(s,X(s),u(s)) - \hat{\sigma}(s,\hat{X}(s),\hat{u}(s)))\right]ds\\ 
&+ \int_0^t \sum_{i,j} \int_{\mathbb{R}_0^n} \left(\theta(s,X(s),u(s),z) - \hat{\theta}(s,\hat{X}(s),\hat{u}(s),z)\right)^T \hat{r}_{i,j}(s,z)\nu_{j}(dz)ds \Big]\Biggr]\\
&= E\Biggl[ \int_0^{\infty} \hat{p}(s)^T(b(s,X(s),u(s))-\hat{b}(s,X(s),u(s)))ds\\ 
&+  \int_0^{\infty} (X(s) - \hat{X}(s))^T (-\nabla_x \hat{H}(s,X(s),u(s),p(s),q(s),r(s,\cdot)))ds \\
&+ \int_0^{\infty} \text{tr}\left[\hat{q}(s)^T (\sigma(s,X(s),u(s)) - \hat{\sigma}(s,\hat{X}(s),\hat{u}(s)))\right]ds\\
&+ \int_0^{\infty} \sum_{i,j} \int_{\mathbb{R}_0^n} (\theta(s,X(s),u(s),z) - \hat{\theta}(s,\hat{X}(s),\hat{u}(s),z))^T \hat{r}_{i,j}(s,z)\nu_{j}(dz)ds \Biggr]\\
&= I_{1,2}^{\infty} + J_{1}^{\infty} + I_{1,3}^{\infty} + I_{1,4}^{\infty}.
\end{align*}
Finally, combining the above we get
\begin{align*}
J(u) - J(\hat{u}) &\leq I_{1}^{\infty} - I_{2}^{\infty} - I_{3}^{\infty} - I_{4}^{\infty}\\
&\leq  -J_1^{\infty} - I_{2}^{\infty} - I_{3}^{\infty} - I_{4}^{\infty}\\
&\leq 0.
\end{align*}
This holds for all $u\in \mathcal{A}_{\mathcal{E}}$, so the proof is complete.
\end{proof}

\section{Examples}

\begin{example}[Optimal Consumption Rate Part I]
Let
\[
 J(u) = E\left[\int_0^{\infty}e^{-\rho t} \ln\big({u(t)X(t)}\big)dt\right],
\]
where
\begin{align*}
 dX(t) &= X(t)(\mu(t) - u(t))dt + X(t)\sigma(t)dB(t),\\
 X(0) &= x_0,
\end{align*}
and $\rho \geq 0$.
We have that
\[
X(t) = X_0 \exp\left[\int_0^t [(\mu(s) - u(s)) - \frac{1}{2}\sigma^2(s)]ds + \int_0^t \sigma(s)dB(s) \right].
\]
Then we deal with the problem of maximizing $J(u)$ over all $u(t)\geq 0$.
We have the Hamiliton function takes the form
\[
 H(t,x,u,p,q) = e^{-\rho t}\ln(ux) +x(\mu - u)p + x\sigma q,
\]
so that we get the partial derivatives
\[
 \nabla_x H(t,x,u,p,q) = \frac{e^{-\rho t}}{x} + (\mu -u)p + \sigma q,
\]
and
\[
 \nabla_u H(t,x,u,p,q) = \frac{e^{-\rho t}}{u} - xp,
\]
This gives us that
\begin{align*}
&-dp(t) = \left[ \frac{e^{-\rho t}}{X(t)} + (\mu(t) - u(t))p(t) + \sigma(t)q(t)\right]dt - q(t)dB(t).
\end{align*}
so that
\[
 \hat{u}(t) = \frac{e^{-\rho t}}{\hat{X}(t)\hat{p}(t)}.
\]
Let us try the infinite horizon BSDE with terminal condition $\underset{t \to \infty}{\lim} p(t) = 0$,
\begin{align}
&-dp(t) = \left[ \frac{e^{-\rho t}}{X(t)} + (\mu(t) - u(t))p(t) + \sigma(t)q(t)\right]dt - q(t)dB(t),\label{eqp1}\\
& \underset{t \to \infty}{\lim} p(t) = 0.\label{eqp2}
\end{align}

\begin{lemma}[Solution of infinite horizon linear BSDE with jumps]
Let $A(t),\beta(t)$ and $\alpha(t,\zeta)$ be $\mathcal{F}_t$-predictable processes such that
\begin{align*}
E\left[\int_0^{\infty} \{|A(t)| + \beta^2(t) + \int_{\mathbf{R}}\alpha^2(s,\zeta)\nu(d\zeta)\}dt\right] < \infty,
\end{align*}
and define $\Gamma_{t,s}$ as the solution of the linear SDE
\begin{align*}
d\Gamma_{t,s} &= \Gamma_{t^{-},s}\bigg(A(t)dt + \beta(t)dB(t) + \int_{\mathbb{R}_0^n} \alpha(t,\zeta) \bar{N}(d\zeta,dt)\bigg), s \geq t \geq 0,\\
\Gamma_{t,t} &= 1.
\end{align*}
Let $C(t)$ be a predictable process such that
\begin{align*}
E\left[\int_0^{\infty} \Gamma_{0,s}|C(s)|ds\right] < \infty.
\end{align*}
Then a solution $(Y(t),Z(t),K(t,\zeta))$ of the linear BSDE
\begin{align*}
-dY(t) &= \left[ A(t)Y(t) + Z(t)\beta(t) + C(t) + \int_{\mathbb{R}_0^n} \alpha(t,\zeta) K(t,\zeta) d\nu(\zeta) \right]dt\\
&- Z(t)dB(t) - \int_{\mathbb{R}_0^n}K(t,\zeta)\bar{N}(d\zeta,dt),\\
\lim Y(t) &= 0, t \to \infty,
\end{align*}
is given by
\[
 Y(t) = E\left[\int_t^{\infty}\Gamma_{t,s}C(s)ds|\mathcal{F}_t\right], t \geq 0.
\]
If in additon
\begin{align*}
E\left[\int_0^{\infty} e^{\lambda t}|Y(t)|^2dt\right] < \infty,
\end{align*}
where $\lambda$ as in \eqref{eq:lambda_req}, then $Y(t)$ is the unique solution.
\end{lemma}
\begin{proof}
 By It$\bar{o}$'s Lemma we have that
\begin{align*}
d(\Gamma_{0,t} Y_t) &= -\Gamma_{0,t}C_tdt + \Gamma_{0,t}(Z_t + Y_t \beta_t)dB_t\\
&+ \int_{\mathbb{R}_0^n}\Bigl[Y(t)\alpha(t,\zeta)\Gamma_{0,t} + K(t,\zeta)\Gamma_{0,t}
+ K(t,\zeta)\alpha(t,\zeta)\Gamma_{0,t}\Bigr]\tilde{N}(d\zeta,dt).
\end{align*}
So
\begin{align*}
\Gamma_{0,t} Y_t &+ \int_t^{\infty} \Gamma_{0,s} C_s ds =  \int_t^{\infty}  \Gamma_{0,s}(Z_s + Y_s \beta_s)dB(s)\\
&+ \int_t^{\infty} \int_{\mathbb{R}_0^n} \Bigl[Y(s)\alpha(s,\zeta)\Gamma_{0,s} + K(s,\zeta)\Gamma_{0,s} 
+ K(s,\zeta)\alpha(s,\zeta)\Gamma_{0,s}\Bigr]\tilde{N}(d\zeta,ds).
\end{align*}
By taking expectation we get the desired result. The uniqueness follows from Theorem 3.1.
\end{proof}

From the above lemma we see that the solution of the linear, infinite horizon BSDE \eqref{eqp1} - \eqref{eqp2} is 
\[
\hat{p}(t) = E\left[ \int_t^{\infty} \frac{\hat{\Gamma}_s}{\hat{\Gamma}_t} \frac{e^{-\rho s}}{\hat{X}_s}ds | \mathcal{F}_t \right], 
\]
where
\[
 \hat{\Gamma}_t = e^{\int_0^t [(\mu(s) - u(s)) - \frac{1}{2}\sigma^2(s)]ds + \int_0^t \sigma(s)dB(s) } = \frac{\hat{X}(t)}{x_0}.    
\]
Hence
\[
 \hat{p}(t) = \frac{1}{\rho}e^{-\rho t}\frac{1}{\hat{X}(t)}.
\]
and
\[
\overline{\lim_{t \to \infty}} \hat{p}(t)(X(t) - \hat{X}(t)) \geq \overline{\lim_{t \to \infty}} \hat{p}(t)X(t) \geq 0.
\]
So 
\[
 \hat{u}(t) =  \rho,
\]
is an optimal control.
\end{example}

\begin{example}[Optimal Consumption Rate - part II]
Let
\[
 J(u) = E\left[\int_0^{\infty}e^{-\rho t} \ln\big({u(t)X(t)}\big)dt\right],
\]
where
\begin{align*}
 dX(t) &= X(t)\mu(t)(1 - u(t))dt + X(t)\sigma(t)(1 - u(t))dB(t),\\
 X(0) &= x_0,
\end{align*}
and $\rho \geq 0$.
We have that
\begin{align*}
X(t) &= X_0 \exp\Bigg[\int_0^t [\mu(s)(1 - u(s)) - \frac{1}{2}\sigma^2(s)(1 - u(t))^2]ds\\
&+ \int_0^t \sigma(s)(1 - u(s))dB(s) \Bigg].
\end{align*}
Then we deal with the problem of maximizing $J(u)$ over all $u(t)\geq 0$.
We have the Hamiliton function takes the form
\[
 H(t,x,u,p,q) = e^{-\rho t}\ln(ux) +x\mu(1 - u)p + x\sigma(1-u) q,
\]
so that we get the partial derivatives
\[
 \nabla_x H(t,x,u,p,q) = \frac{e^{-\rho t}}{x} + \mu(1 -u)p + \sigma(1-u) q,
\]
and
\[
 \nabla_u H(t,x,u,p,q) = \frac{e^{-\rho t}}{u} - x\mu p - x\sigma q.
\]
This gives us that
\begin{align*}
&-dp(t) = \left[ \frac{e^{-\rho t}}{X(t)} + \mu(t)(1 - u(t))p(t) + \sigma(t)(1 - u(s))q(t)\right]dt - q(t)dB(t).
\end{align*}
So that
\[
 \hat{u}(t) = \frac{e^{-\rho t}}{\hat{X}(t)(\mu\hat{p}(t) + \sigma \hat{q}(t))}.
\]
Let us try the infinite horizon BSDE with terminal condition\\
 $\underset{t \to \infty}{\lim} p(t) = 0$, so that
\begin{align}
-dp(t) &= \left[ \frac{e^{-\rho t}}{X(t)} + \mu(t)(1 - u(t))p(t) + \sigma(t)(1 - u(s))q(t)\right]dt \notag\\
&- q(t)dB(t),\label{eq:p1}\\
 \underset{t \to \infty}{\lim} p(t) &= 0.\label{eq:p2}
\end{align}
From the above lemma we see that the solution of the linear, infinite horizon BSDE \eqref{eq:p1} - \eqref{eq:p2} is 
\[
\hat{p}(t) = E\left[ \int_t^{\infty} \frac{\hat{\Gamma}_{0,s}}{\hat{\Gamma}_{0,t}} \frac{e^{-\rho s}}{\hat{X}_s}ds | \mathcal{F}_t \right], 
\]
where
\begin{align*}
 \hat{\Gamma}_t &= \exp \Bigg[\int_0^t \left[\mu(s)(1 - u(s)) - \frac{1}{2}\sigma^2(s)(1 - u(s))^2\right]ds\\
&+ \int_0^t \sigma(s)(1 - u(s))dB(s) \Bigg]\\
&= \frac{\hat{X}(t)}{x_0}. 
\end{align*}
Hence
\[
 \hat{p}(t) = \frac{1}{\rho}e^{-\rho t}\frac{1}{\hat{X}(t)}.
\]
and
\[
\overline{\lim_{t \to \infty}} \hat{p}(t)(X(t) - \hat{X}(t)) \geq \overline{\lim_{t \to \infty}} \hat{p}(t)X(t) \geq 0.
\]
Since
\begin{align*}
d(e^{-\rho t}\frac{1}{X(t)}) &= e^{-\rho t}\frac{1}X(t)dt - e^{-\rho t}\frac{1}{X(t)} (\mu(t) - u(t))dt\\ 
&+ e^{-\rho t}\frac{1}{X(t)} \sigma^2(t)dt + e^{-\rho t}\frac{1}{X(t)}\sigma(t)dB(t),
\end{align*}
we must have that
\[
 \hat {q}(t) = \frac{1}{\rho}e^{-\rho t}\frac{1}{\hat{X}(t)} \sigma(t).
\]
So 
\[
 \hat{u}(t) =  \frac{\rho}{\mu + \sigma},
\]
is an optimal control.
\end{example}

\begin{example}[Optimal consumption rate - part III]
As above, let
\[
 J(u) = E\left[\int_0^{\infty}e^{-\rho t} \ln\big({u(t)X(t)}\big)dt\right].
\]
But add a jump part
\begin{align*}
 dX(t) &= X(t)(\mu(t) - u(t))dt + X(t)\sigma(t)dB(t) + X(t)\int_{\mathbb{R}_0} \theta(t) z \tilde{N}(dz,dt)\\
 X(0) &= x_0,
\end{align*}
and we also add the assumption that we only know a subset of the information given by the market available at time t, represented by $\mathcal{E}_t \subset \mathcal{F}_t$. Let $\rho \geq 0$, be a random variable adapted to $\mathcal{F}_t$.
Then we deal with the problem of maximizing $J(u)$ over all $u(t)\geq 0$.
We have
\[
 H(t,x,u,p,q,r) = e^{-\rho t}\ln(ux) +x(\mu - u)p + x\sigma q + x\int_{\mathbb{R}_0} \theta(t) z r(t,z) \nu(dz)
\]
\[
 \nabla_x H(t,x,u,p,q,r) = \frac{e^{-\rho t}}{x} + (\mu -u)p + \sigma q + \int_{\mathbb{R}_0} \theta(t) z r(t,z) \nu(dz),
\]
\[
 \nabla_u H(t,x,u,p,q,r) = \frac{e^{-\rho t}}{u} - xp
\]

and
\begin{align*}
-dp(t) &= [ \frac{e^{-\rho t}}{X(t)} + (\mu(t) - u(t))p(t) + \sigma(t)q(t) + \int_{\mathbb{R}_0} \theta(t) z r(t,z) \nu(dz)]dt\\
&- q(t)dB(t)- \int_{\mathbb{R}_0} \theta(t) z \tilde{N}(dz,dt), \label{eqp1}\\
\underset{t \to \infty}{\lim} p(t) &= 0.
\end{align*}
If we maximice
\[
 E[H(t,\hat{X}(t),u,\hat{p}(t),\hat{q}(t), \hat{r}(t,\cdot))| \mathcal{E}_t],
\]
we get that
\begin{align*}
\nabla_u E[H(t,\hat{X}(t),u,\hat{p}(t),\hat{q}(t))| \mathcal{E}_t] &=  E[\nabla_u H(t,\hat{X}(t),u,\hat{p}(t),\hat{q}(t))| \mathcal{E}_t]\\
&=  E[\frac{e^{-\rho t}}{u} - \hat{X}(t)\hat{p}(t) | \mathcal{E}_t].
\end{align*}
So that
\[
 \hat{u}(t) = E[\frac{e^{-\rho t}}{\hat{X}(t)\hat{p}(t)}| \mathcal{E}_t].
\]
The solution of the linear, infinite horizon BSDE \eqref{eqp1} - \eqref{eqp2} is (see \cite{BSDE})
\[
\hat{p}(t) = E\left[ \int_t^{\infty} \frac{\hat{\Gamma}_s}{\hat{\Gamma}_t} \frac{e^{-\rho s}}{\hat{X}_s}ds | \mathcal{F}_t \right], 
\]
where
\begin{align*}
&d\hat{\Gamma}_t = X(t)(\mu(t) - u(t))dt + X(t)\sigma(t)dB(t) + X(t^{-})\int_{\mathcal{R}_0} \theta(t) z \tilde{N}(dz,dt), \\
&X(0) = 1. 
\end{align*}
So 
\[
 \hat{\Gamma}_t = \frac{\hat{X}(t)}{x_0}.
\]
Hence
\[
 \hat{p}(t) = \frac{1}{x_0 \hat{\Gamma}_t}\frac{1}{\rho}e^{-\rho t} = \frac{1}{\rho}e^{-\rho t}\frac{1}{\hat{X}(t)}.
\]
Therefore we have that
\[
\overline{\lim_{t \to \infty}} \hat{p}(t)(X(t) - \hat{X}(t)) = \overline{\lim_{t \to \infty}} \hat{p}(t)X(t) \geq 0.
\]
So 
\[
 \hat{u}(t) =  E[ \rho,| \mathcal{E}_t]
\]
is an optimal control.
\end{example}

\begin{example}[Optimal Portfolio Selection With Consumption]
For this example let us look at a market with two investment possibilities:
\begin{enumerate}
 \item A bond or bank account
\[
 dZ_0(t) = \rho Z_0(t)dt.
\]
\item A stock
\[
 dZ_1(t) = \mu Z_1(t)dt + \sigma Z_1(t) dB(t).
\]
\end{enumerate}
Let $(Y_0,Y_1)$ denote the amount the agent has invested in the bonds and stocks repectively at time t.
Consider then $u(t,\omega) = u(t)$, the fraction of the wealth invested in the stocks, e.g. 
\[
 u(t) = \frac{Z_1(t)}{Z_0(t) + Z_1(t)}.
\]
Further let $\lambda(t,\omega) = \lambda(t)$ be the consumption rate relative to the wealth so that the investor controls
\[
 c(t) = (u(t),\lambda(t)).
\]
Then let
\[
 J^{\lambda,u}(s,z) = E^{s,z}\left[\int_0^{\infty}e^{-\delta(s+t)} \frac{(\lambda(t)X(t))^{\gamma}}{\gamma}\right],
\]
be a performance functional, where
\[
 dX(t) = X(t)\left[(\rho + u(t)(\mu - \rho) -\lambda(t))dt + \sigma u(t)dB(t)\right],
\]
and $\rho \geq 0$.
We have that
\[
X(t) = x_0\exp\left[\int_0^t [\rho + u(s)(\mu - \rho) - \lambda(s) - \frac{1}{2}\sigma^2 u^2 ]ds + \int_0^t \sigma u(s)dB(s)\right].
\]
Then we want to maximize $J^{u,\lambda}(s,t)$ over all $ l = (u(t), \lambda(t))$, $ \lambda \geq 0$.
We have that
\[
 H(t,x,l,p,q) = e^{-\delta(s+t)}\frac{(\lambda(t)X(t))^{\gamma}}{\gamma} + x(\rho + u(s)(\mu - \rho) - \lambda)p + x\sigma u q,
\]
so that
\[
 \nabla_x H(t,x,l,p,q) =  e^{-\delta(s+t)} \lambda^{\gamma}x^{\gamma-1} + (\rho + u(\mu - \rho) - \lambda)p +  \sigma u q.
\]
Further, we also have
\[
 -dp(t) = [e^{-\delta(s+t)} \lambda^{\gamma}(t)X^{\gamma-1}(t) + (\rho + u(t)(\mu - \rho) - \lambda(t))p + \sigma u(t) q ]dt - q dB(t).
\]
and
\[
 \nabla_u H(t,x,l,p,q) = (\mu - \rho)xp + x\sigma q,
\]
\[
 \nabla_{\lambda} H(t,x,l,p,q) = e^{-\delta(s+t)} (\lambda(t))^{\gamma-1} X^{\gamma} - x p.
\]
So that
\[
 q(t) = -\frac{(\mu - \rho)}{\sigma}p(t),
\]
and
\[
 \hat{\lambda} = \frac{1}{x} p^{\frac{1}{\gamma -1}} e^{\frac{\delta(s+t)}{\gamma -1}}.
\]
Then 
\begin{align*}
dp(t) &=  - e^{\frac{\delta(s+t)}{\gamma -1}}\frac{1}{X}p^{\frac{\gamma}{\gamma - 1}}(t)dt - [\rho + u(t)(\mu - \rho) -  \frac{1}{X}(t) p^{\frac{1}{\gamma -1}}e^{\frac{\delta(s+t)}{\gamma -1}}]p(t)dt\\
 &+ \sigma u(t) \frac{(\mu - \rho)}{\sigma}p(t) dt - \frac{(\mu - \rho)}{\sigma}p(t) dB(t)\\
&=  - \rho p(t)dt - \frac{(\mu - \rho)}{\sigma}p(t) dB(t).
\end{align*}
So to ensure that the requirement
\[
 E[ \underset{t \to \infty}{\overline{\lim}} \hat{p}(t)(X(t) - \hat{X}(t)) ] \geq 0,
\]
is satisfied we need that
\[
 E[\underset{t \to \infty}{\overline{\lim}} -  \hat{p}(t)\hat{X}(t) ] \geq 0.
\]
Since
\[
 \hat{p}(t) = (\hat{\lambda}(t)\hat{X}(t))^{(\gamma -1)} e^{\delta(s+t)},
\]
we see that 
\[
-\hat{p}(t)\hat{X}(t) = \hat{\lambda}^{(\gamma -1)}(t)\hat{X}^{\gamma}(t) e^{\delta(s+t)}.
\]
So, by considering
\[
 \hat{\lambda} = \frac{1}{x} p^{\frac{1}{\gamma -1}} e^{\frac{\delta(s+t)}{\gamma -1}},
\]
we try to let 
\[
 p^{\frac{1}{\gamma - 1}}(t) = X(t) K e^{B t},
\]
for some constants $K$ and $B$.It is now clear that
\begin{align*}
 d(p^{\frac{1}{\gamma - 1}}(t)) &= p^{\frac{1}{\gamma - 1}}(t)\frac{1}{\gamma -1}(-\rho dt - \frac{(\mu - \rho)}{\sigma} dB(t))\\
&+ p^{\frac{1}{\gamma - 1}}\frac{1}{2}\frac{1}{\gamma -1}\frac{2-\gamma}{\gamma -1}\frac{(\mu - \rho)^2}{\sigma^2}dt.
\end{align*}
On the other hand we have that
\begin{align*}
d(X(t)K e^{B t}) &= B X(t)K e^{B t} dt  + X(t)K e^{B t}[\rho + u(t)(\mu -\rho) - K e^{B t} e^{\frac{\delta(s+t)}{\gamma -1}} ]dt\\
&+ X(t)K e^{B t} \sigma u(t)dB(t).
\end{align*}
Consider
\[
 \hat{u}(t) = -\frac{(\mu - \rho)}{\sigma^2 (\gamma -1)},
\]
and
\begin{align*}
K &= e^{-Bt}e^{-\frac{\delta(s+t)}{\gamma -1}}[B +  \frac{\gamma \rho}{\gamma -1}- \frac{1}{2}\gamma \frac{(\mu - \rho)^2}{\sigma^2 (\gamma -1)^2}].
\end{align*}
For K to be independent of $t$, we must have $B= -\frac{\delta}{\gamma -1}$,
which gives us
\begin{align*}
K &= \left[ -\frac{\delta}{\gamma -1} +  \frac{\gamma \rho}{\gamma -1}- \frac{1}{2}\gamma \frac{(\mu - \rho)^2}{\sigma^2 (\gamma -1)^2}\right].
\end{align*}
With this $K$ and
\[
 \hat{u}(t) = -\frac{(\mu - \rho)}{\sigma^2 (\gamma -1)}
\]
we can conclude that we have 
\[
 p^{\frac{1}{\gamma - 1}}(t) = X(t) K e^{B t}.
\]
It is now clear that
\[
 \hat{\lambda}(t) = K\\
= \hat{\lambda}.
\]
which gives us that
\begin{align*}
\hat{p}(t)\hat{X}(t) &= X^{\gamma}(t) K^{\gamma}(t) e^{\gamma B t}\\
&=  X^{\gamma}(t) \hat{\lambda}^{\gamma},
\end{align*}
so that
\begin{align*}
 -\hat{p}(t)\hat{X}(t) &=-e^{-\frac{\gamma \delta(s+t)}{\gamma -1}}\hat{\lambda}^{\gamma}x^{\gamma}_0 e^{\gamma \int_0^t \rho  - \frac{(\mu - \rho)^2}{\sigma^2 (\gamma - 1)} - \hat{\lambda} - \frac{(\mu - \rho)^2}{\sigma^2 (\gamma - 1)^2}]ds -  \gamma\int_0^t \frac{(\mu - \rho)}{\sigma (\gamma - 1)} dB(s)}\\
&=-e^{-\frac{\gamma \delta(s+t)}{\gamma -1}}_0 \hat{\lambda}^{\gamma}e^{\gamma \rho t - \gamma\frac{(\mu - \rho)^2}{\sigma^2 (\gamma - 1)^2}t - \gamma \hat{\lambda}t - \gamma \frac{(\mu - \rho)^2}{\sigma^2 (\gamma - 1)^2}t -  \gamma\frac{(\mu - \rho)}{\sigma (\gamma - 1)} B(t)}\\
&\geq -e^{-\frac{\gamma \delta(s+t)}{\gamma -1}}\hat{\lambda}^{\gamma}x^{\gamma}_0 e^{- \gamma^2 \frac{(\mu - \rho)^2}{\sigma^2 (\gamma - 1)^2} -  \gamma\frac{(\mu - \rho)}{\sigma (\gamma - 1)} B(t)}.
\end{align*}
If $\delta, \gamma, \rho$ deterministic, then
\begin{align*}
E[ \underset{t \to \infty}{\overline{\lim}} \hat{p}(t)(X(t) - \hat{X}(t)) ] &\geq  -\lim e^{-\delta(s+t)}\hat{\gamma}x^{\gamma} _0 e^{- \gamma^2 \frac{(\mu - \rho)^2}{\sigma^2 (\gamma - 1)^2}}E[ e^{-\frac{(\mu - \rho)}{\sigma (\gamma - 1)} B(t)}]\\
&= 0.
\end{align*}
So we have that $E[ \underset{t \to \infty}{\overline{\lim}} \hat{p}(t)(X(t) - \hat{X}(t)) ] = 0$, which gives us that $(\hat{\lambda},\hat{u})$, where
\[
 \hat{\lambda} = \left[ -\frac{\delta}{\gamma -1} +  \frac{\gamma \rho}{\gamma -1}- \frac{1}{2}\gamma \frac{(\mu - \rho)^2}{\sigma^2 (\gamma -1)^2}\right],
\]
and
\[
 \hat{u} = - \frac{(\mu - \rho)}{\sigma^2 (\gamma - 1)}
\]
is an optimal control.
\end{example}

\section{Necessary Maximum Principle}
To answer the question: if $\hat{u}$ is optimal does it satisfy 
\begin{align}
E\left[ H(t,\hat{X}(t),\hat{u}(t),\hat{p}(t),\hat{q}(t),\hat{r}(t,\cdot)) | \mathcal{E}_t\right] \notag \\
= \max_{u \in U} E\left[ H(t,\hat{X}(t),u,\hat{p}(t),\hat{q}(t),\hat{r}(t,\cdot)) | \mathcal{E}_t\right],
\end{align}
we assume the following two requirements:
\begin{enumerate}
\item[\textbf{A1}] For all $t,h$ such that $0\leq t < t+h \leq T$, all $i=1, \ldots,k$ and for all bounded $\mathcal{E}_t$-measurable $\alpha = \alpha(\omega)$, the control $\beta(s):= (0,\ldots,\beta_i(s),0,\ldots,0) \in U \subset \mathbb{R}^{k}$ with
\begin{align*}
\beta(s) := \alpha_i \mathbf{1}_{[t,t+h]}(s),
\end{align*}
belongs to $\mathcal{A}_{\mathcal{E}}$.
\item[\textbf{A2}] For all $u, \beta \in \mathcal{A}_{\mathcal{E}}$ with $\beta$ bounded, there exists $\delta > 0$ such that $\hat{u} + \epsilon \beta \in \mathcal{A}_{\mathcal{E}}$ for all $\epsilon \in (-\delta,\delta)$.\\

Given $u, \beta \in \mathcal{A}_{\mathcal{E}}$ with $\beta$ bounded, define the process $Y(t) = Y^{(u,\beta)}(t)$ by
\[
 Y(t) = \frac{d}{d\epsilon} X^{\hat{u} + \epsilon \beta }(t)|_{\epsilon=0} = (Y_1(t),..., Y_n(t))^T.
\]
Notice that $Y(0)=0$ and 
\[
 dY_i(t) = \lambda_i(t)dt + \sum_{j=1}^n \xi_{ij}(t)dB_j(t) + \sum_{j=1}^n \int_{\mathbb{R}_0^n} \zeta_{ij}(t,z)\tilde{N}_j(dz,dt),
\]
where
\begin{align*}
 \lambda_i(t) &=   \nabla_x b_i(t,X(t),u(t))^TY(t) + \nabla_u b_i(t,X(t),u(t))^T\beta(t),\\
 \xi_{ij}(t) &=    \nabla_x \sigma_{ij}(t,X(t),u(t))^TY(t) + \nabla_u \sigma_{ij}(t,X(t),u(t))^T\beta(t),\\
 \zeta_{ij}(t,z) &=  \nabla_x \theta_{ij}(t,X(t),u(t))^TY(t) + \nabla_u \theta_{ij}(t,X(t),u(t))^T\beta(t).
\end{align*}

 \end{enumerate}
We can then give a answer to the question.
\begin{theorem}[Partial Information Necessary Maximum Principle]
Suppose $\hat{u} \in \mathcal{A}_{\mathcal{E}}$ is a local maximum for $J(u)$, meaning that for all bounded $\beta \in \mathcal{A}_{\mathcal{E}}$ there exists a $\delta > 0$ such that $\hat{u} + \epsilon \beta \in \mathcal{A}_{\mathcal{E}}$ for all $\epsilon \in (-\delta,\delta)$ and
\[
 h(\epsilon) := J(\hat{u} + \epsilon \beta), \epsilon \in (-\delta,\delta)
\]
is maximal at $\epsilon =0$.
Suppose there exists a solution $(\hat{p}(t),\hat{q}(t),\hat{r}(t,z))$ to the adjoint equation
\begin{align*}
d\hat{p}(t) &= - \nabla_x H(t,\hat{X}(t),\hat{u}(t),\hat{p}(t),\hat{q}(t),\hat{r}(t,\cdot))dt + \hat{q}(t)dB(t)\notag\\
&+ \int_{\mathbb{R}_0^n}\hat{r}(z,t)\tilde{N}(dz,dt),
\end{align*}
and
\[
0 \leq E\left[ \overline{\lim_{t \to \infty}} [ \hat{p}(t)^{T}(X(t) - \hat{X}(t)) ]\right] < \infty,
\]
for all $u \in \mathcal{A}_{\mathcal{E}}$ and $p(t)Y(t,\epsilon)$ converges as $t \to \infty$, uniformly in $\epsilon$, where $Y(t,\epsilon) := \frac{\partial}{\partial \epsilon} X^{\hat{u} + \epsilon \beta}$.
Moreover assume that if $\hat{Y}(t) = Y^{(\hat{u},\beta)}(t)$, with corresponding coefficients $\hat{\lambda}_i$, $\hat{\xi_{ij}}$, $\hat{\zeta_{ij}}$, we have
\[
 E\left[ \hat{Y}(t)^T [\hat{q}\hat{q}^T(t) + \int_{\mathbb{R}_0^n} \hat{r}\hat{r}^T(t,z)\nu(dx)]\hat{Y}(t)dt\right] < \infty,
\]
and
\[
E\left[ \int_0^{\infty} \hat{p}^T(t) [\hat{\xi}\hat{\xi}^T(t,\hat{X}(t),\hat{u}(t)) + \int_{\mathbb{R}_0^n} \hat{\zeta} \hat{\zeta}^T(t,\hat{X}(t),\hat{u}(t),z)\nu(dz)]\hat{p}(t)dt\right] < \infty.
\]

Then $\hat{u}$ is a stationary point for $E[H|\mathcal{E}]$ in the sense that for all $t \geq 0$,
\[
 E[ \nabla_u H(t,\hat{X}(t),\hat{u}(t), \hat{p}(t),\hat{q}(t),\hat{r}(t,\cdot)) | \mathcal{E}_t] = 0.
\]

\end{theorem}

\begin{proof}
Since
\[
0 \leq E\left[ \overline{\lim_{t \to \infty}} [ \hat{p}(t)^{T}(X(t) - \hat{X}(t)) ]\right],
\]
we have that
\[
  E\left[ \overline{\lim_{t \to \infty}} [ \hat{p}(t)^{T}X^{\hat{u} + \epsilon\beta}(t)] \right] \geq  E\left[ \overline{\lim_{t \to \infty}} [  \hat{p}(t)^{T}X^{\hat{u}}(t)) ]\right],
\]
for all $\beta  \in \mathcal{A}_{\mathcal{E}}$ for some $\epsilon$. Define
\[
 g(\epsilon) =  \overline{\lim_{t \to \infty}} [ \hat{p}(t)^{T}X^{\hat{u} + \epsilon\beta}(t)],
\]
so that
\[
 Eg(\epsilon) \geq Eg(0),
\]
for all $\beta \in \mathcal{A}_{\mathcal{E}}$. This means that
\[
\frac{d}{d \epsilon}(Eg(\epsilon))_{\epsilon = 0} = 0.
\]
So
\begin{align*}
0 &= \frac{\partial}{\partial \epsilon}(E\left[ \overline{\lim_{t \to \infty}} [ \hat{p}(t)^{T}X^{\hat{u} + \epsilon\beta}(t)] \right])|_{\epsilon=0}\\
&= E\left[ \frac{\partial}{\partial \epsilon}(\overline{\lim_{t \to \infty}} [ \hat{p}(t)^{T}X^{\hat{u} + \epsilon\beta}(t)])|_{\epsilon=0} \right]\\
&= E\left[ \overline{\lim_{t \to \infty}} [ \hat{p}(t)^{T}\frac{\partial}{\partial \epsilon}(X^{\hat{u} + \epsilon\beta}(t))|_{\epsilon=0}] \right].
\end{align*}
the interchanging of the limit w.r.t. the derivative operator holds for uniform limits with uniform convergence of the derivative. Interchanging derivative and integration is justified if
 \[
\left| \frac{\partial}{\partial \epsilon}(\overline{\lim_{t \to \infty}} [ \hat{p}(t)^{T}X^{\hat{u} + \epsilon\beta}(t,\omega)])|_{\epsilon=0} \right| \leq F(\omega),
 \]
for some integrable function $F$. Now let
\[
 h(\epsilon) = J(\hat{u} + \epsilon \beta),
\]
so that we have
\begin{align*}
0 &= h'(0)\\
&= E\Bigg[ \int_0^{\infty} \Big\{\nabla_x f(t,\hat{X}(t),\hat{u}(t))^T \frac{d}{d\epsilon} X^{\hat{u} + \epsilon \beta }(t)|_{\epsilon=0} 
+  \nabla_u f(t,\hat{X}(t),\hat{u}(t))^T \beta(t)\Big\} dt\\
&+  \overline{\lim_{t \to \infty}} [ \hat{p}(t)^{T}\frac{d}{d\epsilon}(X^{\hat{u} + \epsilon\beta}(t))|_{\epsilon=0}] \Bigg].
\end{align*}
Using It$\bar{o}$'s Lemma we get
\begin{align*}
&E\left[ \overline{\lim_{t \to \infty}} [ \hat{p}(t)^{T} \frac{d}{d\epsilon}(X^{\hat{u} + \epsilon \beta}(t))|_{\epsilon=0} ]\right]\\
&= E\Bigg[ \int_0^{\infty} \Big\{\hat{p}(t)\Big[\nabla_x b(t,\hat{X}(t),\hat{u}(t))^T \frac{d}{d\epsilon} X^{\hat{u} + \epsilon \beta }(t)|_{\epsilon=0}
+ \nabla_u b(t,\hat{X}(t),\hat{u}(t))^T \beta(t)\Big]^T\\
&+ \frac{d}{d\epsilon} X^{\hat{u} + \epsilon \beta }(t)|_{\epsilon=0}(-\nabla_x H(t,\hat{X}(t),\hat{u}(t), \hat{p}(t),\hat{q}(t),\hat{r}(t,\cdot))\\
&+ q(t)(\nabla_x \sigma(t,\hat{X}(t),\hat{u}(t))^T\frac{d}{d\epsilon} X^{\hat{u} + \epsilon \beta }(t)|_{\epsilon=0}  + \nabla_u \sigma (t,\hat{X}(t),\hat{u}(t))^T \beta(t)\\
&+ \hat{r}(t,z) (\nabla_x \theta(t,\hat{X}(t),\hat{u}(t))^T \frac{d}{d\epsilon} X^{\hat{u} + \epsilon \beta }(t)|_{\epsilon=0} + \nabla_u \theta(t,\hat{X}(t),\hat{u}(t))^T \beta(t) \nu(dz) \Big\}dt \Bigg].
\end{align*}
Since
\begin{align*}
\nabla_u H(t,x,u,p,q,r) &= \nabla_u f(t,x,u) + \nabla_u b(t,x,u)p(t) + \nabla_u \sigma(t,x,y)q(t)\\
&+ \int_{\mathbb{R}_0^n} \nabla_u \theta(t,x,u,z)r(t,z)\nu(dz) ,
\end{align*}
and
\begin{align*}
\nabla_u H(t,x,u,p,q,r) &= \nabla_x f(t,x,u) + \nabla_x b(t,x,u)p(t) + \nabla_x \sigma(t,x,y)q(t)\\
&+ \int_{\mathbb{R}_0^n} \nabla_x \theta(t,x,u,z)r(t,z)\nu(dz),
\end{align*}
we have
\begin{align*}
0 &= E\Bigg[ \int_0^{\infty} \Big\{\nabla_u f(t,\hat{X}(t),\hat{u}(t)) +  \nabla_u b(t,\hat{X}(t),\hat{u}(t))\hat{p}^T +  \nabla_u \sigma(t,\hat{X}(t),\hat{u}(t))\hat{q}^T\\
&+  \hat{r} \nabla_u \theta (t,\hat{X}(t),\hat{u}(t))\beta(t))\Big\}dt\Bigg]\\
&= E\left[ \int_0^{\infty} \nabla_u H(t,\hat{X}(t),\hat{u}(t), \hat{p}(t),\hat{q}(t),\hat{r}(t,\cdot))^T \beta(t) dt\right].
\end{align*}
Define
\[
 \beta(s) := \alpha \mathbf{1}_{[t,t+h]}(s).
\]
Then
\begin{align*}
E\left[ \int_t^{t+h} \nabla_u H(t,\hat{X}(t),\hat{u}(t), \hat{p}(t),\hat{q}(t),\hat{r}(t,\cdot))^T \alpha(t) dt\right] =0.
\end{align*}
Differentiating with respect to $h$ at $h=0$ gives
\begin{align*}
E\left[ \nabla_u H(t,\hat{X}(t),\hat{u}(t), \hat{p}(t),\hat{q}(t),\hat{r}(t,\cdot))^T \alpha\right] = 0.
\end{align*}
Since this holds for all $\mathcal{E}$ measurable $\alpha$, we have that
\begin{align*}
E\left[ \nabla_u H(t,\hat{X}(t),\hat{u}(t), \hat{p}(t),\hat{q}(t),\hat{r}(t,\cdot))^T \alpha | \mathcal{E}\right] = 0,
\end{align*}
which proves the theorem.
\end{proof}

\bibliographystyle{plain}
\bibliography{references}

\end{document}